\documentclass[11pt,twoside]{amsart}
\usepackage{amsmath}
\usepackage{amssymb}
\usepackage{amsthm}
\usepackage{latexsym}
\usepackage{amscd}
\usepackage{xypic}
\xyoption{curve}
\usepackage{ifthen}
\usepackage{hyperref}
\usepackage{graphicx}
\usepackage{enumerate}

 \def\cocoa{{\hbox{\rm C\kern-.13em o\kern-.07em C\kern-.13em o\kern-.15em A}}}

\usepackage{xcolor}

\newtheorem{theorem}{Theorem}[section]
\newtheorem{lemma}[theorem]{Lemma}
\newtheorem{proposition}[theorem]{Proposition}
\newtheorem{corollary}[theorem]{Corollary}

\theoremstyle{definition}

\newtheorem{definition}[theorem]{Definition}

\newtheorem{question}[theorem]{Question}

\newcommand {\Hom}{\mathcal{H}\kern -0.25ex{\mathit om}}
\newcommand {\Ext}{\mathcal{E}\kern -0.25ex{\mathit xt}}

\newcommand {\rk}{\mathrm{rk}}

\newcommand {\ext}{\mathrm{Ext}}
\newcommand {\Hilb}{\mathcal{H}\kern -0.25ex{\mathit ilb\/}}

\newcommand {\bP}{\mathbb{P}}
\newcommand {\ZZ}{\mathbb{Z}}

\newcommand{\cE}{{\mathcal E}}
\newcommand{\ccL}{{\mathcal L}}
\newcommand{\cF}{{\mathcal F}}
\newcommand{\cM}{{\mathcal M}}

\newcommand{\cO}{{\mathcal O}}

\newcommand{\cI}{{\mathcal I}}

\newcommand{\Pic}{\operatorname{Pic}}

\def\p#1{{\bP^{#1}}}

\title[Ulrich bundles on Geometrically ruled surfaces]{Theta divisors and Ulrich bundles\\ on Geometrically ruled surfaces}

\subjclass[2010]{Primary 14J60; Secondary 14J26}

\keywords{Vector bundle, Ulrich bundle, geometrically ruled surface.}

\author[M. Aprodu, G. Casnati, L. Costa, R.M. Mir\'o-Roig, M. Teixidor i Bigas]{Marian Aprodu, Gianfranco Casnati, Laura Costa, Rosa Maria Mir\'o-Roig, Montserrat Teixidor i Bigas}
\thanks{The first author was partly supported by a grant of Ministery of Research and Innovation, CNCS - UEFISCDI, project number PN-III-P4-ID-PCE-2016-0030, within PNCDI III. The second author is a member of GNSAGA group of INdAM and is partially supported by the framework of PRIN 2015 \lq Geometry of Algebraic Varieties\rq, cofinanced by MIUR. The third and fourth authors have been partially supported by the grant MTM2016-78623-P}

\address{Facultatea de Matematic\u a \c si Informatic\u a, Universitatea din Bucure\c sti, Str. Academiei 14, 010014 Bucure\c sti, ROMANIA \& Institutul de Matematic\u a \lq\lq Simion Stoilow\rq\rq\ al Academiei Rom\^ane, Calea Grivi\c tei 21, Sector 1, 010702 Bucure\c sti, ROMANIA}
\email{marian.aprodu@fmi.unibuc.ro \& marian.aprodu@imar.ro}

\address{Dipartimento di Scienze Matematiche, Politecnico di Torino, c.so Duca degli Abruzzi 24, 10129 Torino, ITALY}
\email{gianfranco.casnati@polito.it}

\address{Facultat de Matem\`atiques i Inform\`{a}tica,
Departament de Matem\`{a}tiques i Inform\`{a}tica, Gran Via de les Corts Catalanes
585, 08007 Barcelona, SPAIN } \email{costa@ub.edu}

\address{Facultat de Matem\`atiques i Inform\`{a}tica,
Departament de Matem\`{a}tiques i Inform\`{a}tica, Gran Via de les Corts Catalanes
585, 08007 Barcelona, SPAIN } \email{miro@ub.edu}

\address{Mathematics Department, Tufts University, 503 Boston Avenue, Medford MA 02155, USA} \email{montserrat.teixidoribigas@tufts.edu}

\begin{document}

\begin{abstract}
We consider the following question: for which invariants $g$ and $e$ is there a geometrically ruled surface   $S \rightarrow C$  over a curve $C$ of genus $g$ with  invariant $e$ such that $S$ is the support of an Ulrich line bundle with respect to a very ample line bundle?
A surprising relation between the existence of certain proper Theta divisors on some moduli spaces of vector bundles on $C$ with the existence of Ulrich line bundles on $S$ will be the key to completely solve the above question. The relation is realized by translating the vanishing conditions characterizing Ulrich line bundles to specific geometric conditions on the symmetric powers of the defining vector bundle of a given ruled surface. This general principle leads to some finer existence results of Ulrich line bundles in particular cases.
Another focus is on the rank two case where, with very few exceptions, we show the existence of large families of special Ulrich bundles on arbitrary polarized ruled surfaces.
\end{abstract}

\maketitle

\section{Introduction}
Let $X\subseteq\p N$ be a smooth projective variety of dimension $n$ and set $\cO_X(h):=\cO_{\p N}(1)\otimes\cO_X$.
A   vector bundle $\cF$ on $X$ is an  {\sl Ulrich bundle with respect to $\cO_X(h)$},  if
$$
h^i\big(X,\cF(-ih)\big)=h^j\big(X,\cF(-(j+1)h)\big)=0
$$
for each $i>0$ and $j<n$.
For the other equivalent definitions as well as a study of the properties of Ulrich bundles,
 we refer the interested reader to the papers by D. Eisenbud, F.-O. Schreyer and J. Weyman \cite{ES03} and by M. Casanellas and R. Hartshorne \cite{C--H2}.

Ulrich bundles come in pairs, if $\cF$ is Ulrich then so is its \emph{Ulrich dual} $\cF^*((n+1)h+K_X)$, see \cite{A--C--MR}. \emph{Special Ulrich  bundles} are Ulrich self-dual rank-two bundles, \cite{ES03}. Note that, if $\mathcal L$ is an Ulrich line bundle and $\mathcal L'$ is its Ulrich dual,
    then any extension of $\mathcal L'$ by $\mathcal L$ is a special Ulrich bundle.

\medskip

The existence of Ulrich bundles of low rank on a surface $X$, reflects important geometric properties of $X$.
For instance,  if $X$ supports an Ulrich line bundle, then its associated Cayley-Chow form is linear determinantal
 and if $X$ supports special Ulrich rank two bundles then the Cayley-Chow form of $X$ is linear pfaffian (see \cite{Bea} and \cite{ES03}).

\medskip

In this paper we are interested in low rank Ulrich bundles on geometrically ruled surfaces.
Recall that if  $C$ is a smooth curve of genus $g$, then a rank $2$ bundle $\cE$ on $C$ is called {\sl normalized}
 if $h^0\big(C,\cE\big) >0$ and $h^0\big(C,\cE(\frak v)\big)=0$
 for each divisor $\frak v$ on $C$ of negative degree.
  We denote by $\pi\colon S:=\Bbb P(\cE) \rightarrow C$ the geometrically ruled surface defined by a normalized $\cE$.
    Note that this is not restrictive as the ruled surface defined by a vector bundle $\cE$ is isomorphic to  the surface defined by the vector bundle $\cE\otimes \mathcal L^{-1}$
    for any line subbundle $\mathcal L\subseteq\cE$ of maximal degree.   We set  $\frak e:=\bigwedge^2\cE$ and we define {\sl the invariant $e$ of $S$} as the number $e:=-\deg(\frak e)$.

The Picard group of $S$ is generated by $\pi^*\Pic(C)$ and by the class of any effective divisor $C_0$ corresponding to a non-zero section in $H^0\big(S,\cO_S(1)\big)$,
which is isomorphic to $H^0\big(C,\cE\big)$ by the projection formula.
Following \cite{Ha}; Chapter V, Notation 2.8.1, if $\frak b$ is a divisor on $C$ we will write $\frak b f$ instead of $\pi^*\frak b$.
Thus the class of each divisor $D$ on $S$ can be written uniquely as $aC_0+\frak b f$.
For instance,  the canonical divisor $K_S$ on $S$ is in the class $-2C_0+(\frak k+\frak e)f$, $\frak k$ being the canonical divisor on $C$.

The intersection pairing on $S$ is given by $C_0^2=-e$, $C_0 f=1$, $f^2=0$, and we recall that
$$
e=\min\{\ D^2\ \vert\ \text{$D$ is an effective divisor on $S$ with $Df=1$}\ \}.
$$
M. Nagata proved that $e\ge-g$  in \cite{Na}. Moreover, once the curve $C$ is fixed, it is well known that each value satisfying such an inequality is actually attained by some geometrically ruled surface on $C$ (see \cite{Ha}, Theorem V.2.12, Exercise V.2.5 and the references therein).

\vspace{3mm}

In this setting, it is natural to state the  following question:

\begin{question}
\label{q0}
Let $\pi\colon S:=\Bbb P(\cE) \rightarrow C$ be a geometrically ruled surface and consider  $\cO_S(h)$ a very ample line bundle on $S$.
\begin{itemize}
\item[(a)] Are there Ulrich line bundles on $S$ with respect to $\cO_S(h)$?
\item[(b)] More generally, what is the minimal rank of Ulrich vector bundles on $S$?
\end{itemize}
\end{question}

In  \cite{A--C--MR}, a subset of the authors prove that if $e>0$ and $h:=aC_0+\frak b f$ is very ample, then the minimal rank $r$ of an Ulrich bundle
with respect to $\cO_S(h)$ is $r=1$ if and only if $a=1$  and it is $r=2$ for $a=2$.
Moreover, they prove that if $a\ge3$ and either  $g\le1$ or some  additional conditions  on the numbers $a$, $\deg(\frak b)$, $g$, $e$  are satisfied,
 then one  still has $r=2$.
In particular, for $e>0$ and $a>1$ this gives a negative answer to Question \ref{q0} (a).

\vspace{3mm}

In this paper,  we deal with the case  $e\le0$. From the properties of ruled surfaces, we know that $-g \leq e < 0$ (see \cite{Ha}; Chapter V, Exercise 2.5).  We will show that $\frac{(a-1)e}{2} \in \ZZ$ is a necessary condition for the existence, see Proposition \ref{pExistence}.

Our first result is the following statement which summarizes Theorems \ref{tMain} and \ref{tExistenceR} giving a positive answer to Question \ref{q0} (a).

\medbreak
\noindent {\bf Theorem A.}
{\it Let $C$ be a curve of genus $g$, $\cE$ a normalised rank $2$ bundle on $C$ with $e\le 0$ and $h:=aC_0+\frak b f$ a very ample divisor on $S\cong\Bbb P(\cE)$.

\begin{enumerate}
\item If $(a-1)e=0$, then there are Ulrich line bundles on $S$ with respect to $\cO_S(h)$.
\item If $g=1$ and $e<0$, then there are Ulrich line bundles on $S$ with respect to $\cO_S(h)$ if and only if $a$ is odd.
\item If $a=2$ and $e<0$, then there are Ulrich line bundles on $S$ with respect to $\cO_S(h)$ if and only if $e$ is even.
\item If $a=3$, $e<0$ and $g=2$, then there are Ulrich line bundles on $S$ with respect to $\cO_S(h)$.
\item If $a=3$ and  $C$ is general in its moduli space, then there are Ulrich line bundles on $S$ with respect to $\cO_S(h)$.
\end{enumerate}}
\medbreak

When $a \ge 2$ and $e<0$ the picture seems to be very intricate. As shown in Proposition \ref{pExistence} the description of Ulrich line bundles is strictly related to the existence of suitable generic vanishing results for symmetric powers of rank two bundles on curves.
For this reason we formulate the following related question:

\begin{question}
 \label{mainquestion} Let $a$, $g$, and $e$ be integers such that $-g \leq e < 0$, $a \geq 2$ and $\frac{(a-1)e}{2} \in \ZZ$.
 Is there a geometrically ruled surface   $S \rightarrow C$  over a curve $C$ of genus $g$ with  invariant $e$ such that $S$ is the support of an Ulrich line bundle with respect to
$ \cO_S(aC_0+\frak b f)$?
\end{question}

Relating the existence of Ulrich line bundles with the existence of proper Theta divisors on some moduli spaces of vector bundles on $C$ we will be able to give a positive answer to Question \ref{mainquestion} as follows (see Theorem \ref{respq1}).

\medbreak
\noindent {\bf Theorem B.}
{\it Let $a$, $g$, and $e$ be integers such that $-g \leq e < 0$, $a \geq 2$ and $\frac{(a-1)e}{2} \in \ZZ$. Then there exist a geometrically ruled surface   $S \rightarrow C$  over a curve $C$ of genus $g$ with  invariant $e$ such that $S$ is the support of an Ulrich line bundle with respect to
$ \cO_S(aC_0+\frak b f)$.}
\medbreak

\vspace{3mm}

 We then  focus on Question \ref{q0} (b). We conclude the paper by showing the existence of large families of rank two Ulrich vector bundles with respect to $\cO_S(aC_0+\frak b f)$.  For $a=1$ they can be constructed
 as an extensions of Ulrich line bundles. For  $a \geq 2$ and some mild conditions on  $\frak b$ we construct them in Theorems \ref{tACMR} and \ref{tStable}.

\medbreak

{\bf Notation:}
Throughout this note we will work on an algebraically closed field $k$ of characteristic $0$ and $\p N$ will denote the projective space over $k$ of dimension $N$.
The words curve and surface will always denote projective smooth connected objects. In several places, we shall mix the multiplicative notation for line bundles and the additive notation for divisors.

\section{Ulrich line bundles on ruled surfaces}
\label{sRank1}
The goal of this section is to determine the existence of Ulrich line bundles on a geometrically ruled surface $S$ with negative invariant $e$
and in particular to prove Theorem A stated in the introduction.
We start by recalling some useful facts.

If $D:=tC_0+\frak d f$ with $t\ge0$ is a divisor on $S$, then Lemma V.2.4, Exercises III.8.3 and III.8.4 of \cite{Ha} imply
\begin{equation}
\label{Projection}
h^i\big(S,\cO_S(D)\big)=h^i\big(C,(S^t\cE)(\frak d)\big)
\end{equation}
where $S^t\cE$ stands for the $t$-th symmetric power of $\cE$.

On the other hand, since $\cE$ is normalized, there is an everywhere non-zero section in $H^0\big(C,\cE\big)$ defining the  exact sequence
\begin{equation*}
\label{seqExtension}
0\longrightarrow\cO_C\longrightarrow\cE\longrightarrow\cO_C(\frak e)\longrightarrow0
\end{equation*}
(see the proof of \cite{Ha}, Theorem V.2.12). Notice that  such an extension corresponds to an element $\xi\in H^1\big(C,\cO_C(-\frak e)\big)$.

Thus there also exists an exact sequence of the form (see the proof of Lemma 7.6 of \cite{E--H})
\begin{equation*}
\label{seqSymmetric}
0\longrightarrow S^{t-1}\cE\longrightarrow S^t\cE\longrightarrow\cO_C(t\frak e)\longrightarrow0
\end{equation*}
where $t\ge1$. Due to its construction, such an extension depends on the choice of $\xi$ and $t$.

Easy induction on $t$ using \eqref{Projection} yields
\begin{gather}
\label{BoundLine}
h^0\big(C,(S^t\cE)(\frak d)\big)\ge h^0\big(C,\cO_C(\frak d)\big),\\
\label{BoundSymmetric}
h^0\big(C,(S^t\cE)(\frak d)\big)\le \sum_{i=0}^th^0\big(C,\cO_C(\frak d+i\frak e)\big).
\end{gather}
for each divisor $\frak d$ on $C$.

The following lemma is a particular case of \cite{Cs4}; Corollary 2.2.

\begin{lemma}
\label{lUlrichLine}
Let $C$ be a curve of genus $g$, $\cE$ a normalized rank $2$ vector bundle on $C$ and $h:=aC_0+\frak b f$ a very ample divisor on $S\cong\Bbb P(\cE)$.

The line bundle $\cO_S(D)$ is Ulrich with respect to $\cO_S(h)$ if and only if $h^0\big(S,\cO_S(D-h)\big)=h^0\big(S,\cO_S(2h+K_S-D)\big)=0$ and
\begin{equation}
\label{eqLineBundle}
D^2=2(h^2-1+g)+DK_S,\qquad Dh=\frac12(3h^2+hK_S).
\end{equation}
\end{lemma}

The following result is the main part of \cite{A--C--MR}; Theorem 2.1  where we  set
$$
d(a,g,e):=g-1+\frac{(a-1)e}2.
$$
For the reader's benefit we repeat here the proof under our current assumptions, mainly with no restrictions on $e$.


\begin{proposition}
\label{pExistence}
Let $C$ be a curve of genus $g$, $\cE$ a normalized rank $2$ vector bundle on $C$ and $h:=aC_0+\frak b f$ a very ample divisor on $S\cong\Bbb P(\cE)$.

There is an Ulrich line bundle on $S$ with respect to $\cO_S(h)$ if and only if  $d(a,g,e)\in \mathbb{Z}$ and there
exist divisors $\frak u\in\Pic^{d(a,g,e)}(C)$ satisfying
\begin{equation}
\label{frakU}
h^0\big(C,(S^{a-1}\cE)(\frak u)\big)=0.
\end{equation}
The Ulrich line bundles on $S$ are exactly the ones of the form
$$
\cO_S((2a-1)C_0+(\frak b+\frak u)f),
$$
and their Ulrich duals
$$
\cO_S((a-1)C_0+(2\frak b+\frak k+\frak e-\frak u)f),
$$
for each $\frak u$ on $C$ satisfying condition \eqref{frakU} above.
In particular, if $(a-1)e$ is odd, then there are no Ulrich bundles on $S$ with respect to $\mathcal{O}_S(h)$.
\end{proposition}
\begin{proof}
Assume that $S$ supports an Ulrich line bundle $\mathcal L\cong\cO_S(a_1C_0+\frak b_1f)\otimes\cO_S(h)$. Thus its Ulrich dual $\cM:=\cO_S(3h-2C_0+(\frak k+\frak e) f)\otimes\mathcal L^\vee$ is also an Ulrich bundle.

In particular, if $\cM\cong\cO_S(a_2C_0+\frak b_2f)\otimes\cO_S(h)$, then
$$
(a_1+a_2)C_0+(\frak b_1+\frak b_2)f=h-2C_0+(\frak k+\frak e) f=(a-2)C_0+(\frak b+\frak k+\frak e)f.
$$
Since both $\mathcal L$ and $\cM$ are assumed to be Ulrich with respect to $\cO_S(h)$, it follows that $\chi(\mathcal L(-h))=\chi(\cM(-h))=0$. Thus, a direct computation via the Riemann--Roch theorem as in the first part of the proof of
 \cite{A--C--MR};  Theorem 2.1 yields the vanishing

$$
(a_i+1)\left(\deg(\frak b_i)-d(a_i+1,g,e)\right)=0, \qquad i=1,2.
$$
If $\deg(\frak b_i)=d(a_i+1,g,e)$, for $i=1,2$, then the equality $\frak b_1+\frak b_2=\frak b+\frak k+\frak e$ yields $\deg(\frak b)=ae/2$. It would follow $h^2=0$, trivially contradicting the ampleness of $\cO_S(h)$.
If $a_1=a_2=-1$, then $a=0$, again a contradiction.
Thus we can assume $a_1=-1$ and $a_2=a-1$, whence $\deg(\frak b_2)=d(a,g,e)$ and
\begin{gather*}
\mathcal L\cong\cO_S(-C_0+(\frak b+\frak k+\frak e-\frak b_2)f)\otimes\cO_S(h),\qquad
\mathcal M\cong\cO_S((a-1)C_0+\frak b_2f)\otimes\cO_S(h).
\end{gather*}

Let $D:=(a-1)C_0+{\frak b}_{2f+h}$,
 so that $\cO_S(D)\cong \cM $  and $\cO_S(3h+K_S-D) \cong \mathcal L$.
  The divisor $D$ satisfies the equalities \eqref{eqLineBundle} and $h^0\big(S,\cO_S(2h+K_S-D) \big)=h^0\big(S,\mathcal L (-h)\big)=0$.

On the other hand, according to (\ref{Projection}),
$$
h^0\big(S,\cO_S(D-h)\big)=h^0\big(C,(S^{a-1}\cE)\otimes\cO_C(\frak b_2)\big).
$$

Thus, taking $\frak u:=\frak b_2$ the statement follows from Lemma \ref{lUlrichLine}.
\end{proof}

For the proof of the following result see \cite{A--C--MR}, Theorem 2.1.

\begin{proposition}
\label{pA--C--MR}
Let $C$ be a curve of genus $g$, $\cE$ a normalized rank $2$ vector bundle on $C$ and $h:=aC_0+\frak b f$ a very ample divisor on $S\cong\Bbb P(\cE)$.

If $e>0$, then there are Ulrich line bundles on $S$ with respect to $\cO_S(h)$ if and only if $a=1$.
\end{proposition}

Thanks to Proposition \ref{pExistence} we are able to extend the above proposition to the case $e\le0$.

\begin{theorem}
\label{tMain}
Let $C$ be a curve of genus $g$, $\cE$ a normalized rank $2$ bundle on $C$ and $h:=aC_0+\frak b f$ a very ample divisor on $S\cong\Bbb P(\cE)$.

If $(a-1)e=0$, then there exist two families of dimension $g$ of Ulrich line bundles with respect to $\cO_S(h)$.
\end{theorem}
\begin{proof}
We have $d(a,g,e)= g-1$ because $(a-1)e=0$. Thus,  the set $\frak U\subseteq\Pic^{d(a,g,e)}(C)$ of line bundles $\cO_S(\frak u)$ such that $h^0\big(C,\cO_C(\frak u)\big)=0$ is open and non-empty,
because it is the complement of the theta divisor $W^1_{g-1}(C)$. Trivially $\dim(\frak U)=\dim(\Pic^{g-1}(C))=g$.

In particular, if Ulrich line bundles $\mathcal U$ on $S$ exist, their characterization in Proposition \ref{pExistence} means that they form two families according to
whether $c_1(\mathcal U) f$ is $2a-1$ or $a-1$. Both these families have the same dimension $\dim(\frak U)=g$.

The rest of the proof is devoted to showing that a general divisor in $\frak U$ actually satisfies condition \eqref{frakU}.

The case $a=1$ is trivial due to the non-emptiness of $\frak U$. So,  we can  restrict ourselves to the case $a\ge2$ and $e=0$. Such  case is very easy to handle when $g=0$ (see \cite{Cs4}, Example 2.3). Therefore, we will assume $g \geq 1$.

Notice that for each $i\ge0$,  we have $\deg(\frak u+i\frak e)=d(a,g,0)=g-1$. Thus, for all $0 \leq i \leq a-1$ and for each $\frak u$ with  $\cO_S(\frak u+i\frak e)\in \frak U=\Pic^{g-1}(C)\setminus W^1_{g-1}(C)$

$$h^0\big(C,\cO_C(\frak u+i\frak e)\big)=0 .$$
 Inequality \eqref{BoundSymmetric}  yields $h^0\big(C,(S^{a-1}\cE)(\frak u)\big)=0$ for such an $\frak u$. Thus the statement follows from Proposition \ref{pExistence}.
\end{proof}

\vspace{3mm}
When $a\ge2$ and $e<0$ the picture is much more intricate.
 In order to prove the existence of Ulrich line bundles in this setting, in this section we will relate their existence  to the so called Raynaud's condition and in the next section to  the existence of suitable theta divisors.
 Let us introduce Raynaud's condition.

\vspace{3mm}

 Let $\cF$ be a vector bundle of rank $r$ on a curve $C$ of genus $g$. Riemann--Roch's Theorem for $\cF$ gives
$$
\chi(\cF)=h^0\big(C,\cF\big)-h^1\big(C,\cF\big)=r(\mu(\cF)+1-g).
$$
If $\frak v\in\Pic^0(C)$, one has $\chi(\cF)=\chi(\cF(\frak v))=h^0\big(C,\cF(\frak v)\big)-h^1\big(C,\cF(\frak v)\big)$. The integer $h^0\big(C,\cF(\frak v)\big)$ is a function of $\frak v$, but there exists a non-empty open subset $\frak V\subseteq\Pic^0(C)$ where it takes a constant value, say $h^0_\mathrm{gen}(\cF)$.

Assume now $\deg(\cF)\le0$. If $r=1$, then $h^0_\mathrm{gen}(\cF)=0$, or, in other words, $h^0_\mathrm{gen}(\cF)=\max\{\ 0, \chi(\cF)\ \}$.

\begin{definition}
Let $C$ be a curve of genus $g$ and $\cF$ a vector bundle on $C$. We say that $\cF$ satisfies condition $(\star)$ if and only if $h^0_\mathrm{gen}(\cF)=\max\{\ 0, \chi(\cF)\ \}$.
\end{definition}

Condition $(\star)$ is also known in the literature as Raynaud's condition.

\vspace{3mm}

The relation between Raynaud's condition and Ulrich line bundles is the following

\begin{lemma}
\label{lExistence}
Let $C$ be a curve of genus $g$, $\cE$ a normalized rank $2$ vector bundle on $C$ and $h:=aC_0+\frak b f$ a very ample divisor on $S\cong\Bbb P(\cE)$.

Then there is an Ulrich line bundle on $S$ with respect to $\cO_S(h)$ if and only if $d(a,g,e)\in\mathbb{Z}$ and there exist divisors $\frak u\in\Pic^{d(a,g,e)}(C)$ such that $(S^{a-1}\cE)(\frak u)$ satisfies condition $(\star)$.
\end{lemma}
\begin{proof}
For each divisor $\frak u\in\Pic^{d(a,g,e)}(C)$ (if any, i.e. if $(a-1)e$ is even), we have
$$
\deg((S^{a-1}\cE)(\frak u))=\deg(S^{a-1}\cE)+a\deg(\frak u)=a(g-1),
$$
and hence $\mu((S^{a-1}\cE)(\frak u))=g-1\ge0$. Thus,  the equality $h^0_\mathrm{gen}((S^{a-1}\cE)(\frak u))=0$ is equivalent to condition $(\star)$.
\end{proof}

Note that semistable bundles of slope precisely $g-1$ are of special interest in view of condition $(\star)$.
Indeed, as pointed out in Raynaud's work, semistable bundles of smaller slope automatically satisfy this condition, and hence this is a borderline case.
For a geometric phenomenon related to the bundles of slope $g-1$ we refer to Proposition 1.8.1 in~ \cite{Ray}.

The above proposition together with the results proved in \cite{Ray} allow us to prove the existence of Ulrich line bundles in several cases.
In particular, if $g=1$, then we are able to give a complete answer concerning the  existence of Ulrich line bundles.

\begin{theorem}
\label{tExistenceR}
Let $C$ be a curve of genus $g$, $\cE$ a normalized rank $2$ bundle on $C$ with $e<0$ and $h:=aC_0+\frak b f$ a very ample divisor on $S\cong\Bbb P(\cE)$.

\begin{enumerate}
\item If $g=1$, then there are Ulrich line bundles on $S$ with respect to $\cO_S(h)$ if and only if $a$ is odd.
\item If $a=2$, then there are Ulrich line bundles on $S$ with respect to $\cO_S(h)$ if and only if $e$ is even.
\item If $a=3$ and $g=2$, then there are Ulrich line bundles on $S$ with respect to $\cO_S(h)$.
\item If $a=3$ and  $C$ is general in its moduli space, then there are Ulrich line bundles on $S$ with respect to $\cO_S(h)$.
\end{enumerate}
\end{theorem}
\begin{proof}
First of all notice that since $\cE$ is a normalized rank two bundle of degre $-e>0$, it is $\mu$-semistable and the same holds for any of its symmetric powers and their twists. Let us start with the proof of assertion (1). As $-g\le e<0$, if $g=1$, then the hypothesis $e<0$ forces $e=-1$.
Thus if $a$ is even there are no Ulrich line bundles on $S$ due  to Proposition \ref{pExistence}.
If $a$ is odd, then $d(a,g,e)$ is an integer and hence $\Pic^{d(a,g,e)}(C)$ is non-empty.  Thus the statement follows from \cite{Ray}, Corollaire 1.7.3 and Lemma \ref{lExistence} because  $(S^{a-1}\cE)(\frak u)$ is $\mu$-semistable, as noted above.

Let us prove assertion (2).  If $e$ is odd, then there are no Ulrich line bundles on $S$ due to Proposition \ref{pExistence}. If $e$ is even, then the rank $3$ vector bundle $(S^2\cE)(\frak u)$ is $\mu$-semistable and the statement follows from Lemma \ref{lExistence} and \cite{Ray}, Proposition 1.6.2.

Assertion (3), follows from Lemma \ref{lExistence} and \cite{Ray}, Corollaire 1.7.4 due to the fact that $(S^3\cE)(\frak u)$ is a rank $4$ $\mu$-semistable vector bundle. Finally, since $(S^3\cE)(\frak u)$ is a rank $4$ $\mu$-semistable vector bundle, (4) follows  from \cite{Ray}, section 2.5.
\end{proof}

\section{Ulrich line bundles and theta divisors}
\label{sRank1Theta}
The goal of this section is to determine the existence of geometrically ruled surfaces $S$ with negative invariant $e$ supporting Ulrich line bundles
and in particular to prove Theorem B stated in the introduction.

In some sense Raynaud's condition is related to the existence of theta divisors on moduli spaces of semistable vector bundles. We will prove the existence of proper theta divisors of some symmetric powers of rank two vector bundles $\cE$ on $C$ and this will give us the existence, under some generic conditions, of Ulrich line bundles on geometrically ruled surfaces $S\cong\Bbb P(\cE)$. In particular, we will be able to give a positive answer to Question \ref{mainquestion}.

Let us recall the definition of theta divisors. Denote by $U(r,d)$  the moduli space of rank $r$ semistable vector bundles $\cE$
on $C$ of degree  $\deg(\cE)=d$.

\begin{definition}  Let $\cF$ be a vector bundle of degree $d$ and rank $r$ on $C$. Denote by $j$ the greatest common divisor of $d, r$.
Then $d=jd_1, r=jr_1$.
We define
\[ \Theta _{\cF}=\{\ \cF_1\in U(r_1, r_1(g-1)-d_1) \text{ such that } h^0\big(C,\cF\otimes \cF_1\big)>0\ \}. \]
\end{definition}
 If $\cF_1$ is vector bundle of rank $r_1$ and degree $r_1(g-1)-d_1$, then $\chi (\cF\otimes \cF_1)=0$.
It is expected that for a generic $\cF$ and  generic
$\cF_1$, the space of sections of $\cF\otimes \cF_1$ will be zero and that $\Theta _{\cF}$ will be a divisor of the moduli space.
If this is the case, $ \Theta _{\cF}$ is called a {\em theta divisor}. For a generic $\cF$ of rank $r$ and degree $d$, it is known that $ \Theta _{\cF}$ is a divisor of  $U(r_1, r_1(g-1)-d_1)$ (see \cite{Ray} Prop. 1.8(1))
 but this is not true for every $\cF$.
It has been shown that for some values of $r, d$, there exist vector bundles, sometimes even infinite families of $\cF$ for which
$$\Theta _{\cF}=U(r_1, r_1(g-1)-d_1)$$
(see \cite{Po}).
\vspace{3mm}

In view of Proposition \ref{pExistence} we have the following characterization  of the existence of Ulrich line bundles in terms of the existence of theta divisors.

\begin{proposition}
\label{theta}
Let $C$ be a curve of genus $g$, $\cE$ a normalized rank $2$ bundle on $C$  and $h:=aC_0+\frak b f$ a very ample divisor on $S\cong\Bbb P(\cE)$. Then, there is an Ulrich line bundle on $S$ with respect to $\cO_S(h)$ if and only if $d(a,g,e) \in \ZZ$ and  $ \Theta _{S^{a -1}\cE}$ is a proper divisor of
$\Pic^{d(a,g,e)}(C)$.
\end{proposition}
\begin{proof}
We have $d(a,g,e)\in\ZZ$ if and only if $(a-1)e$ is even, hence
\[ \Theta _{S^{a -1}\cE}=\{\ \ccL\in \Pic^{d(a,g,e)}(C) \text{ such that } h^0\big(C,(S^{a -1}\cE)\otimes \ccL\big)>0\ \}, \]
because  $r=a$ and $d=-a(a-1)e/2$ are the rank and the degree of $S^{a -1}\cE$.
The statement then follows from Proposition \ref{pExistence}.
\end{proof}

Therefore, our next  goal  is to study the existence of theta divisors of symmetric powers of rank two normalized vector bundles on $C$.

\begin{lemma} \label{normsublb}Let $C$ be a curve, $\cE$ a vector bundle of rank two on $C$, $\ccL_1$ a line subbundle of maximal degree on $C$.
Then, $\cE\otimes \ccL_1^{-1}$ is normalized.
\end{lemma}
\begin{proof} By assumption, there exists an injective map $0\to \ccL_1\to \cE$ and therefore also a map $0\to {\mathcal O}_C\to \cE\otimes \ccL_1^{-1}$.
Hence, $h^0(\cE\otimes \ccL_1^{-1})>0$.
Assume now that there is a line bundle $\ccL$ of negative degree such that $h^0(E\otimes  \ccL_1^{-1} \otimes \ccL)>0$,
 then $\ccL_1 \otimes \ccL^{-1}$ is a subsheaf of $\cE$ of degree higher than the degree of $\ccL_1$ contradicting the assumption.
\end{proof}

\begin{proposition}\label{filtracio} Fix integers $d, r, r', \ 0<r'<r$. For a vector bundle $\cE$ of rank $r$ and degree $d$ define
\[ s_{r'}(\cE)=r'd-r\max \{\ \deg \cE'\ | \ \rk \cE'=r', \cE'\subset \cE\ \}.\]
For a fixed $s$ with $0<s\le r'(r-r')(g-1), s\equiv r'd\  (r)$, define
\[ U_{r',s}(r,d)=\{\ \cE\in U(r,d)\ |\ s_{r'}(\cE)=s\ \}. \]
Then $U_{r',s}(r,d)$ is non-empty, irreducible of dimension $r^2(g-1)+1+s-r'(r-r')(g-1)$ and $U_{r',s}(r,d)\subset \overline{U_{r',s+r}(r,d)}$.

\end{proposition}
\begin{proof} If $g=0,1$ there are no integers $s$ satisfying the above restrictions. If $g\ge2$ see  \cite{RT} Theorem 0.1 and Corollary 1.12.
\end{proof}

\begin{corollary} \label{normands}Let $C$ be a curve of genus g, $\cE$ a rank two vector bundle of  degree $d$ on $C$.
Then, if $\cE\in U_{1,s}(2,d)$, there exists a line bundle  $\ccL$ of degree $\frac{d-s}2$ such that $\bar \cE=\cE\otimes \ccL^{-1}$ is normalized.
In particular
\[ 0<s=\deg \bar \cE\le g.\]
\end{corollary}
\begin{proof}
The statement follows from Lemma \ref{normsublb} and  Proposition \ref{filtracio}.
\end{proof}

 \bigskip

\begin{proposition}\label{thetatensor} Let $C$ be a curve of genus g, $\cF$ a rank $r$ vector bundle of degree $d$  and $\ccL$ a line bundle of degree $d'$.
 With the notations above  $\Theta _{\cF} $ is a proper divisor of $U(r_1, r_1(g-1)-d_1)$ if and only if
 $\Theta _{\cF\otimes \ccL} $ is a proper divisor of $U(r_1, r_1(g-1-d')-d_1)$.
\end{proposition}
\begin{proof}  The map
\[U(r_1, r_1(g-1)-d_1)\to U(r_1, r_1(g-1-d')-d_1), \ \ \cF'\to \cF'\otimes \ccL^{-1}\]
gives a bijection between the two moduli spaces (with inverse $\cF''\to \cF''\otimes \ccL$).
From the definition of the theta locus, $\Theta _{\cF} $ maps to $\Theta _{\cF\otimes \ccL} $ under this map.
So one locus is a divisor if and only if the other is.
\end{proof}

\begin{corollary}\label{thetatensorsym} Let $C$ be a curve of genus g, $\cE$ a rank two vector bundle of degree $d$,  $\ccL$ a line bundle of degree $d'$
and $\beta$ a positive integer.
 \begin{itemize}
 \item[(a)] If $\beta d$ is even,   $\Theta _ {S^{\beta }\cE}$ is a proper divisor of $\Pic^{g-1-\frac {\beta d}2}(C)$ if and only if
 $\Theta _ {S^{\beta }(\cE\otimes \ccL)}$ is a proper divisor of $\Pic^{g-1-\frac {\beta d}2-\beta d'}(C)$.
 \item[(b)] If $\beta d$ is odd,   $\Theta _ {S^{\beta }\cE}$ is a proper divisor of
  $U(2, 2(g-1)-\beta d)$ if and only if
 $\Theta _{S^{\beta} (\cE\otimes \ccL)} $ is a proper divisor of $U(2, 2(g-1)-\beta d-2\beta d')$.
 \end{itemize}
\end{corollary}
\begin{proof}  This follows from Proposition \ref{thetatensor} as $S^{\beta} (\cE\otimes \ccL)=(S^{\beta} \cE)\otimes \ccL^{\beta}$ and $U(1,d)=\Pic^{g-1-\frac {\beta d}2-\beta d'}(C)$.
\end{proof}

  Now we are ready to state the first main result concerning the existence of proper theta divisors.  To this end, we first consider a normalized rank two vector bundle $\cE$ of  even degree $\deg(\cE)=-e=2f$ and $a \geq 2$ an integer. In particular,  the slope of $\cE$ is $f$.
 The bundle $\cF=S^{a-1}\cE$ has rank $a$ and slope  $(a -1)f$.
 Therefore,
\[ \Theta _{S^{a -1}\cE}=\{\ \ccL\in \Pic^{d(a,g,e)}(C) \text{ such that } h^0\big(C,(S^{a -1}\cE)\otimes \ccL\big)>0\ \}. \]

\begin{proposition}
\label{propoeven} Let $C$ be any curve of genus g and  fix an even degree  $-e=2f$ and an integer $s,  0 < s\leq g$.
 Let $\cE$ be a normalized vector bundle generic in the stratum $U_{1,s}(2,2f)$. Then, for all $a$, $ \Theta _{S^{a -1}\cE}$ is a proper divisor of  $ \Pic^{d(a,g,e)}(C)$.
\end{proposition}
\begin{proof} The condition that $ \Theta _{S^{a -1}\cE}$ is a proper divisor  is an open condition in the moduli space of vector bundles.
Using  Proposition \ref{filtracio} it suffices to prove the result for the smallest stratum of the moduli space of vector bundles
corresponding to those bundles with subbundles of the largest degree.
These bundles  are extensions of two line bundles of the same degree.
They  can be deformed to a direct sum $\ccL_1\oplus \ccL_2$  of two (different) generic line bundles each of degree $f$
 (consider the family of extensions of one of the line bundles by the other).
Then $S^{a -1}\cE$ is the set of tensors  in $(\ccL_1\oplus \ccL_2)^{\otimes (a -1)}$ that are invariant under the action of the symmetric group.
The set
\[\{\ \ccL\in \Pic^{d(a,g,e)}(C) \text{ such that } h^0\big(C,\ccL_1^k \otimes \ccL_2^{a -1-k}\otimes \ccL\big)>0\ \}\]
 is a theta divisor in the Jacobian. Moreover, $ \Theta _{S^{a -1}\cE}$ is contained in the union
of these theta divisors as $k$ varies. Thus,   it is still a divisor.
\end{proof}

\begin{corollary} \label{even}
Let $C$ be a curve of genus $g$,  $\cE$  a normalized rank two vector bundle on $C$ generic among those of degree $2f >0$ and
 $h:=aC_0+\frak b f$ a very ample divisor on $S\cong\Bbb P(\cE)$. Then, there is an Ulrich line bundle on $S$ with respect to $\cO_S(h)$.
\end{corollary}
\begin{proof} From Proposition \ref{theta}, we need to check that   $ \Theta _{S^{a -1}\cE}$ is a proper divisor.
From Corollary \ref{normands}, it suffices to do this for a generic point of $U_{1,e}(2,2f)$.
From Proposition \ref{propoeven}, this holds.
\end{proof}

\vspace{3mm}

Let us now  consider a normalized rank two vector bundle $\cE$ of  odd degree $\deg(\cE)=-e=2f+1$ and $a \geq 2$ an integer.

\vspace{3mm}

\begin{proposition}
\label{propoodd} Let $C$ be a generic curve of genus g and $\cE$ be a generic normalized rank two vector bundle of degree $-e=2f+1$ on $C$.
Then for odd $a$, $ \Theta _{S^{a -1}\cE}$ is a proper divisor of  $ \Pic^{d(a,g,e)}(C)$ and
for even $a$, $ \Theta _{S^{a -1}\cE}$ is a proper divisor of  the moduli space $ U(2, 2g-2-(a -1)(2f+1))$.
\end{proposition}
\begin{proof} We start by deforming the curve $C$ to a chain $C_0$ of elliptic curves as follows.  Consider $C_1,\dots, C_g$ generic elliptic curves. Let $P_i, Q_i$ be generic points on $C_i$.
We construct a curve $C_0$ of arithmetic genus $g$ by identifying $P_i$ with $Q_{i-1}, i=2,\dots, g$. Let us now determine a generic normalized rank two  vector bundle $\cE_0$ of degree $2f+1$ on $C_0$. To this end,  take a generic  indecomposable vector bundle $\cE_1$ of rank 2 and degree 1 on $C_1$, a direct sum of two generic line bundles of degree one $\ccL_1^i\oplus \ccL_1^{'i}$ on $C_i, i=3, 5,\dots, 2f+1$; and a direct sum of two generic line bundles of degree zero $\ccL_0^i\oplus \ccL_0^{'i}$ on the remaining components
$C_i, i=2, 4, 6\dots 2f$ and $i=2f+2, 2f+3, \dots, g$. We take the gluing so that $\ccL_0^i, i=2f+2, 2f+3, \dots, g$ glue to each other but the gluings  are generic otherwise.
  One can then find a degree zero line subbundle $\ccL^i_0$ of $\ccL_1^i\oplus \ccL_1^{'i}$ on $C_i, i=3, 5,\dots, 2f+1$ that  glues with  both $\ccL^{i-1}_0$ and  $\ccL^{i+1}_0$ and a degree zero line subbundle of $\cE_{1}$ that glues with $\ccL_0^2$.
  Gluing these degree zero line subbundles on each component produces a degree zero line subbundle of the vector bundle $\cE_0$ on the chain.
One can in fact check that the largest degree of a subbundle of $\cE_0$ on $C_0$ is zero (see \cite{T} for  details). Hence, $\cE_0$ is a generic normalized rank two vector bundle on $C_0$ of degree $2f+1$.

In order to show that the theta divisor of the symmetric power of $\cE_0$ is an actual divisor, it suffices to find  a line bundle $\ccL_0$ on $C_0$ (resp.  vector bundle $\cF_0$) of the appropriate
degrees such  that $(S^{a -1}\cE_0)\otimes \ccL_0$
(resp ($S^{a -1}\cE_0)\otimes \cF_0$) do not have any limit linear section.

Notice that the vector bundle $\cE_0$ that we built on $C_0$ has restriction to $C_i, i=2,\dots, g$ a direct sum of two generic line bundles
$\cE_i=\ccL_j^i\oplus \ccL_j^{'i}$ of the same degree $j=0,1$
and has restriction to $C_1$ a generic indecomposable rank two vector bundle $\cE_{1}$ of  degree $1$.
The $(a -1)$-symmetric power of  $\cE_i, i>1$ is a subsheaf of $\oplus _{m+l=a -1}(\ccL_j^i)^m\otimes(\ccL_j^{'i})^l$.
The $(a -1)$-symmetric power of  $\cE_1$ is a subsheaf of the $(a -1)$-tensor power of $\cE_{ 1}$.
Since $\cE_1$ is a rank two vector bundle, $\cE_{1}^*\cong \cE_{ 1}\otimes \bar \ccL^{-1}$ where $\bar \ccL=\det \cE_{1}$. On the other hand,  it follows from \cite{A}; Lemma 22 that
 $ \cE_{1}\otimes \cE_{1}^*=\oplus_{i=0,\dots,3} \cM_i$ where the $\cM_i$ are the elements in $C_1$ of order 2.

 Therefore,
 \begin{equation} \label{odd1} \otimes ^{a -1} \cE_{ 1}=\begin{cases}  (\oplus_{i=0,\dots,3} \cM_i)^{\otimes k}\otimes \bar \ccL ^k &\text{ if } a =2k+1 \\
  (\oplus_{i=0,\dots,3} \cM_i)^{\otimes k}\otimes \cE_{1}\otimes \bar L ^k &\text{ if } a =2(k+1).    \end{cases}  \end{equation}

 A limit linear series of slope $g-1$ of the tensor product of $S^{a -1}\cE_0$ with an arbitrary  line bundle $\ccL_0$ of degree $g-1-\frac {(2f+1)(a -1)}2$
  (resp. an arbitrary rank two vector bundle $\cF_0$ of  degree $2(g-1)-(2f+1)(a -1)$) would give rise to a section on each component with proper vanishing at the nodes.

  Assume first that  $a $ is odd and therefore $a -1$ is even. According to (\ref{odd1}) and the construction of $\cE_0$,
   \[   (\otimes ^{a -1} \cE_{0}\otimes \ccL_0)_{|C_i}=\begin{cases}  (\oplus_{i=0,\dots,3} \cM_i)^{\otimes k}\otimes \bar \ccL ^k \otimes \ccL_{01} &\text{ if } i=1 \\
 \oplus _{m+l=a -1}(\ccL_j^i)^m \otimes (\ccL_j^{'i})^l \otimes \ccL_{0i} &\text{ if } i>1    \end{cases}  \]
 where $\ccL_{0i}$ denotes the restriction of $\ccL_0$ to $C_i$.

  On the other hand, the restriction $\ccL_{01}$  has degree  $g-1-\frac{a -1}2$ on $C_1$, the restriction $\ccL_{0i}$ has  degree $g-a$ on the components $C_3, C_5,\dots, C_{2f+1}$
 and degree $g-1$  on the remaining components $C_2, C_4, C_6\dots C_{2f}$ and  $C_{2f+2}, C_{2f+3}, C_g$. Moreover, it is generic of the stated degree on each of the components. Notice also that
  $\ccL_j^i, \ccL_j^{'i}, \bar \ccL$ are fixed determined by the generic choice of $\cE_0$. In addition,  the line bundles  $\cM_i$ form a subgroup, so their product is another element $\cM_j$ in this subgroup and their degree is zero. Hence, the degree of $\bar \ccL ^k \otimes  \ccL_{01}$ and of  $\ccL_j^i \otimes \ccL_j^{'i} \otimes \ccL_{0i}$ is $g-1 $.
 By the genericity of $\ccL_{0i}$, the sum of the orders of vanishing of a section of the line bundles $\cM_i\otimes \bar \ccL ^k \otimes \ccL_{01},  \ccL_j^i \otimes \ccL_j^{'i}\otimes \ccL_{0i}$
  at the two nodes of the elliptic curve is at most $g-2$.  In order to have a limit linear series, the order of vanishing at $Q_i$ of the sections on the component $C_i$ and  the order of vanishing at $P_{i+1}$ of the sections on the component $C_{i+1}$
  needs to be at least $g-1$.
  Therefore, the sum $O$ of vanishing orders at the nodes satisfies
  \[   (g-1)^2\le O\le (g-2)g\]
  which is impossible.

  Consider now the case in which $a $ is even and then $\cF_0$ is an arbitrary rank two  vector bundle of degree $2g-2-(2f+1)(a -1)$.
  We take $\cF_{0,1}$ on $C_1$ to be a generic vector bundle of  degree $2g-1-a$, $\cF_{0i}$
     the direct sum of two generic line bundles of degree $g-a$ on the components $C_3, C_5,\dots , C_{2f+1}$
   and the direct sum of two generic line bundles of degree $g-1$  on the remaining components $C_2, C_4, C_6\dots C_{2f}$ and on $C_{2f+2}, C_{2f+3},\dots , C_g$.
   Any two indecomposable vector bundles of rank two and odd degree differ in product with a line bundle (see \cite{A}; corollary to Theorem 7).
   Therefore, there exists a line bundle $\ccL_{0,1}$ of degree $g-1$ such that   $\cF_{0,1}=\cE_{0,1}\otimes \ccL_{0,1}$
   Then,
    \[   (\otimes ^{a -1} \cE_{0}\otimes \cF_0)_{|C_i}=\begin{cases}  (\oplus_{i=0,\dots,3} \cM_i)^{\otimes k+1}\otimes \bar \ccL ^k \otimes \ccL_{01} &\text{ if } i=1 \\
 \oplus _{m+l=a -1}(\ccL_j^i)^m \otimes (\ccL_j^{'i})^l \otimes (\ccL_{0i} \oplus \ccL_{0i}^{'}) &\text{ if } i>1.     \end{cases}  \]
   The same argument as before shows that this cannot have a limit section.

\end{proof}

Because our proof uses a deformation argument to a special kind of curve, we cannot conclude that the result is true for {\em every} curve.
It is likely though that, as in the case of even degree, the result holds for every curve.

\begin{corollary} \label{odd}
Let $C$ be a generic curve of genus $g$,  $\cE$  a normalized rank two vector bundle on $C$ generic among those of degree $-e=2f+1 >0$ and $h:=aC_0+\frak b f$ a very ample divisor on $S\cong\Bbb P(\cE)$ with $a$ odd (i.e. $d(a,g,e) \in \ZZ$). Then, there is an Ulrich line bundle on $S$ with respect to $\cO_S(h)$.
\end{corollary}
\begin{proof} It follows from Proposition \ref{theta} and Proposition \ref{propoodd}
\end{proof}

\vspace{3mm}

Now we are ready to give a positive answer to Question \ref{mainquestion} stated in the introduction

\begin{theorem}
\label{respq1} Let $a$, $g$, and $e$ be integers such that $-g \leq e < 0$, $a \geq 2$ and $\frac{(a-1)e}{2} \in \ZZ$. Then there exist a geometrically ruled surface   $S \rightarrow C$  over a curve $C$ of genus $g$ with  invariant $e$ such that $S$ is the support of an Ulrich line bundle with respect to
$ \cO_S(aC_0+\frak b f)$.
\end{theorem}
\begin{proof} It follows from Corollary \ref{even} and Corollary \ref{odd}
\end{proof}

\section{Stable rank $2$ Ulrich bundles}
\label{sRank2}
In this section, we focus our attention on the existence of special stable rank $2$  Ulrich bundles. The existence is known for $g=0$ and $a\ge2$ see \cite{Cs4}, Theorem 1.2. When $a=1$ and $g=0$ there are no such bundles, actually it is known (see \cite{F--M}, Corollary to Theorem B), that each Ulrich bundle of rank at least $2$ is in this case strictly $\mu$-semistable. Finally, when $g=1$ the existence of $\mu$-stable special Ulrich bundles of rank $2$ is proved in \cite{Cs5}; Theorem 1.2.

\vspace{3mm}

First of all, we study the existence of $\mu$-stable rank $2$ Ulrich bundles on $S$ with respect to $h:=aC_0+\frak b f$ for $a \geq 2$. To this end, we set
$$
u_e:=
\left\lbrace
\begin{array}{ll}
ae\quad&\text{if $e>0$,}\\
e\quad&\text{if $e\le0$,}
\end{array}
\right.
$$

\vspace{3mm}

With this notation we have the following result:

\vspace{3mm}

\begin{theorem}
\label{tACMR}
Let $C$ be a curve of genus $g\ge1$, $\cE$ a normalized rank $2$ bundle on $C$ and $h:=aC_0+\frak b f$ a very ample divisor on $S\cong\Bbb P(\cE)$.

If $a=2$ or $a\ge3$ and
$$
\deg(\frak b)>\max\left\{ \ \frac{(a-3)(g-1)+e a}{2},(g-1)+u_e ,\frac{e(3a+1)}{6}+\frac{2g}{3} \ \right\},
$$
then, for each general 0-dimensional subscheme $Z\subseteq S$ of degree $(a-1)(\deg(\frak b)-e a/2)$ and each general $\frak v\in\Pic^{g-1}(C)$, there exist $\mu$-stable special Ulrich bundles $\cF$ of rank $2$ with respect to $\cO_S(h)$ fitting into the exact sequence
$$
0\longrightarrow\cO_S(aC_0+(\frak b+\frak v)f)\longrightarrow \cF\longrightarrow \cI_{Z\vert S}((2a-2)C_0+(2\frak b+\frak k+\frak e-\frak v)f)\longrightarrow 0.
$$
\end{theorem}
\begin{proof}
The same argument used in \cite{A--C--MR}, Proposition 3.3 and Theorem 3.4  can be extended almost verbatim to prove that for any $e$ these extensions define rank $2$ special Ulrich bundles. The unique point where that proof must be slightly modified is when the ampleness of  $\cO_S(aC_0+(\frak b-\frak v)f)$ is needed. In order to infer such an ampleness we need $\deg(\frak b)-g+1>ae/2$ if $e \leq 0$ and $\deg(\frak b)-g+1>ae$ if $e > 0$. Both inequalities are true since
$$
\deg(\frak b) > (g-1)+u_e.
$$

We need to prove that bundles in the statement are $\mu$-stable. They are certainly $\mu$-stable in the range $e>0$ and $a\ge2$. In fact, Ulrich vector bundles are always $\mu$-semistable and they can only be destabilized by Ulrich line bundles which, in this range, do not exist by \cite{A--C--MR}; Theorem 2.1.

In order to prove the statement for $a \geq 2$ and $e \leq 0$, we compute the dimension of the entire family $\bP$ of bundles defined in the statement
and we  estimate the dimension of the subfamily of strictly $\mu$-semistable ones.

First, we check that a general $\cF$ belongs to one single extension. To this end, it suffices to prove that
\[
h^0\big(S,\mathcal{I}_{Z|S}((a-2)C_0+(\frak b+\frak k-2\frak v+\frak e))\big)=0
\]
for general $Z$ and $\frak v$. Observe that
\begin{align*}
h^0\big(S,\mathcal{O}_S((a-2)C_0+(\frak b+\frak k-2\frak v+\frak e))\big)&=
h^0\big(C,(S^{a-2}\cE)(\frak b+\frak k-2\frak v+\frak e)\big)\le \\
&\le \sum_{i=0}^{a-2}
h^0\big(C,\cO_C(\frak b+\frak k-2\frak v+(i+1)\frak e)\big).
\end{align*}

On the other hand,  $\mathrm{deg}(\frak b+\frak k-2\frak v+(i+1)\frak e)=\mathrm{deg}(\frak b)-(i+1)e>g-1$ by the assumption $\deg(\frak b) > (g-1)+e$. Hence $h^1\big(C,\cO_C(\frak b+\frak k-2\frak v+(i+1)\frak e)\big)$ will be zero for all $i$ with $0 \leq i \leq a-2$, and for a generic choice of $\frak v$. Thus,  applying the Riemann--Roch theorem, we obtain
\[
h^0\big(S,\mathcal{O}_S((a-2)C_0+(\frak b+\frak k-2\frak v+\frak e))\big)\le (a-1)\mathrm{deg}(\frak b)-\frac{a(a-1)e}{2}+(a-1)(1-g)\le \deg(Z).
\]
Hence, for a generic choice of $Z$, we have $h^0\big(S,\mathcal{I}_{Z|S}((a-2)C_0+(\frak b+\frak k-2\frak v+\frak e))\big)=0$.

The bundles $\cF$  are parameterized by a projective bundle $\bP$ with typical fibre the projectivization of
\begin{align*}
\ext^1_S\big(\cI_{Z\vert S}((2a-2)C_0+(2\frak b&+\frak k+\frak e-\frak v)f),\cO_S(aC_0+(\frak b+\frak v)f)\big)\cong\\
&\cong H^1\big(S, \cI_{Z\vert S}((a-4)C_0+(\frak b+2\frak k+2\frak e-2\frak v)f)\big)
\end{align*}
over an open subset of the product of the Hilbert scheme of subschemes of $S$ of dimension $0$ and degree $(a-1)(\deg(\frak b)-e a/2)$ multiplied by $\Pic^{g-1}(C)$.

Let $a\ge3$. In the proof of \cite{A--C--MR}; Theorem 3.4,  the authors prove that
\begin{equation}
\label{Ample}
\begin{aligned}
h^0\big(S,\cO_S((a-4)C_0+(\frak b&+2\frak k+2\frak e-2\frak v)f)\big)=\\
&=(a-3)\left(g-1+\deg(\frak b)-\frac{ea} 2\right),\\
h^1\big(S,\cO_S((a-4)C_0+(\frak b&+2\frak k+2\frak e-2\frak v)f)\big)=\\
&=h^2\big(S,\cO_S((a-4)C_0+(\frak b+2\frak k+2\frak e-2\frak v)f)\big)=0.
\end{aligned}
\end{equation}

It follows from the first equality \eqref{Ample}  and the hypothesis on $2\deg(\frak b) > (a-3)(g-1)+e a$ that
\begin{equation}
\label{Ideal2}
h^0\big(S, \cI_{Z\vert S}((a-4)C_0+(\frak b+2\frak k+2\frak e-2\frak v)f)\big)=0
\end{equation}
for a general choice of $Z$.

Let $a=2$. Trivially $h^0\big(S,\cO_S(-2C_0+(\frak b+2\frak k+2\frak e-2\frak v)f)\big)=0$. Moreover for $i\ge1$
\begin{align*}
h^i\big(S,\cO_S(-2C_0+(\frak b+2\frak k+2\frak e-2\frak v)f)\big)&=h^{2-i}\big(S,\cO_S((-\frak b-\frak k-\frak e+2\frak v)f)\big)\\
&=h^{2-i}\big(C,\cO_C(-\frak b-\frak k-\frak e+2\frak v)\big)\\
&=h^{i-1}\big(C,\cO_C(\frak b+2\frak k+\frak e-2\frak v)\big).
\end{align*}
We have $\deg(\frak b+2\frak k+\frak e-2\frak v)=\deg(\frak b)-e+2g-2$, thus
\begin{equation}
\label{a=2}
\begin{gathered}
h^1\big(S,\cO_S(-2C_0+(\frak b+2\frak k+2\frak e-2\frak v)f)\big)=\deg(\frak b)-e+g-1,\\
h^2\big(S,\cO_S(-2C_0+(\frak b+2\frak k+2\frak e-2\frak v)f)\big)=0.
\end{gathered}
\end{equation}

In both cases above ($a=2$ and $a\ge3$), the cohomology of the sequence
\begin{align*}
\label{seqStandard}
0\longrightarrow\cI_{Z\vert S}\longrightarrow \cO_S\longrightarrow \cO_Z\longrightarrow 0
\end{align*}
tensored by $\cO_S((a-4)C_0+(\frak b+2\frak k+2\frak e-2\frak v)f)$ and the equalities \eqref{Ample}, \eqref{Ideal2}, \eqref{a=2} yield
$$
h^1\big(S, \cI_{Z\vert S}((a-4)C_0+(\frak b+2\frak k+2\frak e-2\frak v)f)\big)=2\deg(\frak b)-(a-3)(g-1)-ae.
$$

Therefore, since $\bP$ depends on the schemes $Z$, the bundles $\mathfrak u$ and the extension classes, it follows that
\begin{eqnarray}
\label{DimensionFamily}
\dim(\bP)=2(a-1)\left(\deg(\mathfrak b)-\frac{ea}{2}\right)+g+(2\deg(\frak b)-(a-3)(g-1)-ae-1)\\
\nonumber
=2a\deg(\frak b)-ea^2-(a-4)(g-1).
\end{eqnarray}

From now on we will assume that  $\cF$ is strictly $\mu$-semistable. If this is the case, it contains an Ulrich line bundle and, in particular, by Proposition \ref{pExistence} this implies that  $(a-1)e$ is even. Recall also that we are under the assumption $e\le 0$.

Ulrich line bundles are given by Proposition \ref{pExistence}. We cannot have
$$
\cO_S((2a-1)C_0+(\frak b+\frak u)f)\subseteq\cF
$$
for in this case there would necessarily be a non-zero morphism from $\cO_S((2a-1)C_0+(\frak b+\frak u)f)$ to either $\cO_S(aC_0+(\frak b+\frak v)f)$ or $\cO_S((2a-2)C_0+(2\frak b+\frak k+\frak e-\frak v)f)$ which is a contradiction since $H^0(\cO_S((1-a)C_0+(\frak v- \frak u)f))=0$ and $H^0(\cO_S(-C_0+(\frak b+\frak k+\frak e-\frak v-\frak u)f))=0$.

Thus, we necessarily have   $\cO_S((a-1)C_0+(2\frak b+\frak k+\frak e-\frak u)f)\subseteq \cF$ and since their quotient must be also an Ulrich line bundle we get an exact sequence of the form
\begin{equation}
\label{seqSemistable}
0\longrightarrow\cO_S((a-1)C_0+(2\frak b+\frak k+\frak e-\frak u)f)\longrightarrow\cF\longrightarrow\cO_S((2a-1)C_0+(\frak b+\frak u)f)\longrightarrow0.
\end{equation}
By the projection formula,
\begin{align*}
\ext^1_S\big(\cO_S((2a-1)C_0+(\frak b+\frak u)f)&,\cO_S((a-1)C_0+(2\frak b+\frak k+\frak e-\frak u)f)\big) \\
&\cong H^1\big(S,\cO_S(-aC_0+(\frak b+\frak k+\frak e-2\frak u)f)\big) \\
&\cong H^1\big(S,\cO_S((a-2)C_0+(2\frak u-\frak b)f)\big)\\
&\cong H^1\big(C,(S^{a-2}\cE)(2\frak u-\frak b)\big).
\end{align*}

The extensions as in the exact sequence \eqref{seqSemistable} are parameterised by a space of dimension
\begin{align*}
h^1\big(C,(S^{a-2}\cE)&(2\frak u-\frak b)\big)-1=\\
&=(a-1)\left(\deg(\frak b)-\frac{ae}2+1-g\right)-1+h^0\big(C,(S^{a-2}\cE)(2\frak u-\frak b)\big)\\
&\le (a-1)\left(\deg(\frak b)-\frac{ae}2+1-g\right)-1+\sum_{i=0}^{a-2}h^0\big(C,\cO_C(2\frak u-\frak b+i\frak e)\big),
\end{align*}
where the inequality follows from the inequality \eqref{BoundSymmetric}.

We have $(a-1-i)e\le 0$, hence the Riemann--Roch theorem and the Clifford theorem imply
\begin{align*}
h^0\big(C,\cO_C(2\frak u-\frak b+i\frak e)\big)\le g-1+\max\left\{\ (a-1-i)e-\deg(\frak b),\frac{(a-1-i)e-\deg(\frak b)}2+1\ \right\},
\end{align*}
for each general choice of $\frak u$.

Since $e \leq 0$,  $\deg(\frak b)\ge0$ and then
 $$ h^0\big(C,\cO_C(2\frak u-\frak b+i\frak e)\big)\le \frac{(a-1-i)e-\deg(\frak b)}2 + g. $$
 Hence the strictly $\mu$-semistable bundles we are interested in are parameterised by a family of dimension at most
$$
\frac{a-1}2\left(\deg(\frak b)+2-\frac{ae}2\right)-1.
$$

The  dimension of the family of strictly $\mu$-semistable bundles is smaller than the value $\dim(\bP)$ given by the equality \eqref{DimensionFamily} if
\[
\frac{a}{2}\left(3\deg(\frak b)-\frac{e}{2}(3a+1)-2g\right) +4(g-1)+\frac{b}{2}+2>0,
\]
condition which is automatically satisfied, since $3\deg(\frak b)\ge \frac{e(3a+1)}{2}+2g$, by hypothesis.

Since a general $\cF$ is uniquely determined by an extension in $\mathbb P$, we conclude that there are $\mu$-stable Ulrich bundles.
\end{proof}

Finally, we consider slightly different extensions to remove the restriction on $\mathrm{deg}(\frak b)$ for $e=0$.  To this purpose, we will make use of the following characterization of special Ulrich bundles of rank $2$ in our particular setup.

\begin{lemma}
\label{lUlrichSpecial}
Let $C$ be a curve of genus $g$, $\cE$ a normalized rank $2$ vector bundle on $C$ and $h:=aC_0+\frak b f$ a very ample divisor on $S\cong\Bbb P(\cE)$.

A rank $2$ vector bundle $\cF$ on $S$ is a special Ulrich bundle with respect to $\cO_S(h)$ if and only if $h^0\big(S,\cF(-h)\big)=0$ and
\begin{equation}
\label{eqSpecialBundle}
\mathrm{det}(\cF)=\cO_S(3h+K_S),\qquad c_2(\cF)=\frac{1}2(5h^2+3hK_S)+2-2g.
\end{equation}
\end{lemma}
\begin{proof}
For the proof we refer the interested reader to  \cite{Cs4}, Corollary 2.2 or \cite{A--C--MR}, Lemma 3.2.
The proof therein is given under the apparently more restrictive hypothesis that $\cF$ is initialized, i.e. $h^0\big(S,\cF\big)\ne0$ and $h^0\big(S,\cF(-h)\big)=0$.
Only the second vanishing is actually necessary. Indeed the condition $h^0\big(S,\cF\big)\ne0$ follows easily from the Riemann--Roch theorem applied to $\cF$.
\end{proof}

We will end  by proving the existence of rank two Ulrich bundles for $e=0$ without the restriction on  $\mathrm{deg}(\frak b)$.

\begin{theorem}
\label{tStable}
Let $C$ be a curve of genus $g\ge1$, $\cE$ a normalized rank $2$ bundle of degree zero on $C$ and $h:=aC_0+\frak b f$ a very ample divisor on $S\cong\Bbb P(\cE)$ with $a\ge 2$. Let $\alpha:=\left[a/2\right]$ and $\epsilon:=a-2\alpha$.

Then for each general subscheme $Z\subseteq S$ of dimension $0$ of degree $(\alpha+\epsilon)\deg(\frak b)$ and each general $\frak v\in\Pic^{g-1}(C)$, there exist $\mu$-stable special Ulrich bundles $\cF$ of rank $2$ with respect to $\cO_S(h)$ fitting into the exact sequence
$$
0\longrightarrow\cO_S((3\alpha-1+\epsilon) C_0+(\frak b+{\frak v}+\frak e)f)\longrightarrow \cF\longrightarrow \cI_{Z\vert S}((3\alpha-1+2\epsilon) C_0+(2\frak b+\frak k-{\frak v})f)\longrightarrow 0.
$$
\end{theorem}

\begin{proof}
Let $\cO_S({\frak v})\in\frak U=\Pic^{g-1}(C)\setminus W^1_{g-1}(C)$  be a non-effective line bundle on $C$.

We define the following two divisors on $S$
\begin{gather*}
D:=(3\alpha-1+\epsilon) C_0+(\frak b+{\frak v}+\frak e)f,\qquad A:=\epsilon C_0+(\frak b+\frak k-2{\frak v}-\frak e)f.
\end{gather*}
We trivially have $A+2D=3h+K_S$.

We observe that $\deg(\frak b)\ge 3$. In fact,  the restriction of $\cO_S(h)$ to $C_0$ is very ample: since $\pi_{\vert C_0}$ is an isomorphism on $C$ and it follows that
\begin{equation}
\label{BoundPositive}
\deg(\frak b)=hC_0\ge3.
\end{equation}

It is clear that $h^0\big(S,\cO_S(A+K_S)\big)=0$. Thus, by \cite{H--L}; Theorem 5.1.1 there exists a rank $2$ vector bundle  $\cF$  fitting into the  exact sequence
\begin{equation}
\label{seqSpecial}
0\longrightarrow \cO_S(D)\longrightarrow \cF\longrightarrow \cI_{Z\vert S}(A+D)\longrightarrow 0.
\end{equation}

Simple computations show that the equalities \eqref{eqSpecialBundle} are satisfied. We will show that $h^0\big(S,\cF(-h)\big)=0$ and that $\cF$ is $\mu$-stable for a general choice of ${\frak v}$ and $Z$.

We now prove that $h^0\big(S,\cF(-h)\big)=0$. To this purpose we will check that
$$
h^0\big(S,\cO_S(D-h)\big)=0 \quad \mbox{and } h^0\big(S,\cI_{Z\vert S}(D+A-h)\big)=0
$$
for a general choice of ${\frak v}$ and $Z$.

We have $D-h=(\alpha-1)C_0+({\frak v}+\frak e)f$. For a general ${\frak v}$ the inequality \eqref{BoundSymmetric} yields
$$
h^0\big(S,\cO_S(D-h)\big)\le \sum_{i=0}^{\alpha-1}h^0\big(C,\cO_C({\frak v}+(i+1)\frak e)\big).
$$
Since $\deg({\frak v}+(i+1)\frak e)=g-1$, choosing $\frak v$ such that $\frak v+(i+1)\frak e\in \Pic^{g-1}(C)\setminus W^1_{g-1}$ for all $i=0,\ldots,\alpha-1$, we  obtain $h^0\big(S,\cO_S(D-h)\big)=0$.

On the other hand,  $D+A-h=(\alpha-1+\epsilon) C_0+(\frak b+\frak k-{\frak v})f$.  Hence again the inequality \eqref{BoundSymmetric} yields
$$
h^0\big(S,\cO_S(D+A-h)\big)\le \sum_{i=0}^{\alpha-1+\epsilon}h^0\big(C,\cO_C(\frak b+\frak k-{\frak v}+i\frak e)\big).
$$

Since $\mathrm{deg}(\frak b)\ge 3$ (see the inequality \eqref{BoundPositive}) and $\mathrm{deg}(\frak e)=0$, it follows that $\deg({\frak v}-\frak b-i\frak e)\le g-1$ for all $i$, which implies
$$
h^1\big(C,\cO_C(\frak b+\frak k-{\frak v}+i\frak e)\big)=h^0\big(C,\cO_C({\frak v}-\frak b-i\frak e)\big)=0,
$$
for general ${\frak v}$. Thus $h^0\big(C,\cO_C(\frak b+\frak k-{\frak v}+i\frak e)\big)=\deg(\frak b)$ and an easy computation shows
$$
h^0\big(S,\cO_S(D+A-h)\big)\le  (\alpha+\epsilon)\deg(\frak b).
$$
Thus for a general choice of $Z$,  $h^0\big(S,\cI_{Z\vert S}(D+A-h)\big)=0$. In particular, by Lemma \ref{lUlrichSpecial}, $\cF$ is a special Ulrich bundle of rank $2$.

\vspace{3mm}

We now show that a general $\cF$ is $\mu$-stable. If an Ulrich bundle of rank $2$ is not $\mu$-stable, then it contains a proper Ulrich subbundle of rank $1$. We will show that for each Ulrich line bundle $\mathcal L$ on $S$ we have $h^0\big(S,\cF\otimes\mathcal L^\vee\big)=0$. Since $\cF$ fits into the exact sequence \eqref{seqSpecial}, this will follow if we prove that
\begin{equation*}
\label{Vanishing}
h^0\big(S,\cO_S(D)\otimes\mathcal L^\vee\big)=0 \quad \mbox{and }  h^0\big(S,\cI_{Z\vert S}(A+D)\otimes\mathcal L^\vee\big)=0.
\end{equation*}
The Ulrich line bundles on $S$ are described in Proposition \ref{pExistence}. Notice that $\cO_C(\frak u)\in \Pic^{g-1}(C)$ because $\cE$ is a bundle of degree $e=0$.

If
$$
\mathcal L\cong\cO_S((2a-1)C_0+(\frak b+\frak u)f)=\cO_S((4\alpha+2\epsilon-1)C_0+(\frak b+\frak u)f),
$$
then the required vanishings follow by looking at the coefficient of $C_0$ in $\cO_S(D)\otimes\mathcal L^\vee$ and $\mathcal{O}_S(A+D)\otimes\mathcal L^\vee$ respectively. Now let
$$
\mathcal L\cong\cO_S((a-1)C_0+(2\frak b+\frak k-\frak u+\frak e)f)=\cO_S((2\alpha-1+\epsilon) C_0+(2\frak b+\frak k-\frak u+\frak e)f).
$$
The inequality \eqref{BoundSymmetric} implies
\begin{align*}
h^0\big(S,\cO_S(D)\otimes\mathcal L^\vee\big)\le \sum_{i=0}^{\alpha}h^0\big(C,\cO_C(-\frak b+\frak u+{\frak v}-\frak k+i\frak e)\big).
\end{align*}
Trivially
$$
\deg(-\frak b+\frak u+{\frak v}-\frak k+i\frak e)\le-\deg(\frak b)\le -1,
$$
hence $h^0\big(C,\cO_C(-\frak b+\frak u+{\frak v}-\frak k+i\frak e)\big)=0$ which implies that $h^0\big(S,\cO_S(D)\otimes\mathcal L^\vee\big)=0$.
Finally $\cO_S(A+D)\otimes\mathcal L^\vee\cong \cO_S((\alpha+\epsilon) C_0+(\frak u-{\frak v}-\frak e)f)$ hence the inequality \eqref{BoundSymmetric} yields
\begin{align*}
h^0\big(S,\cO_S(A+D)\otimes\mathcal L^\vee\big)\le \sum_{i=0}^{\alpha+\epsilon}h^0\big(C,\cO_C(\frak u-{\frak v}+(i-1)\frak e)\big).
\end{align*}

We have $e=0$, so that  $\deg(\frak u-{\frak v}+(i-1)\frak e)=0$. If $\frak u\ne {\frak v}-(i-1)\frak e$, then
$$
h^0\big(S,\cI_{Z\vert S}(A+D)\otimes\mathcal L^\vee\big)\le h^0\big(S,\cO_S(A+D)\otimes\mathcal L^\vee\big)=0.
$$

If $\frak u= {\frak v}-(i-1)\frak e$, then $\cO_S((\alpha+\epsilon) C_0+(\frak u-{\frak v}-\frak e)f)\cong\cO_S((\alpha+\epsilon) C_0+(-i\frak e)f)$. If, moreover $\frak e\ne \mathcal O_C$
then
\begin{align*}
h^0\big(S,\cO_S(A+D)\otimes\mathcal L^\vee\big)\le \sum_{i=0}^{\alpha+\epsilon}h^0\big(C,\cO_C((i-1)\frak e)\big)=0.
\end{align*}
Finally, if $\frak u= {\frak v}-(i-1)\frak e$ and $\frak e = \mathcal O_C$ then
$$
h^0\big(S,\cO_S(A+D)\otimes\mathcal L^\vee\big)\le \alpha+\epsilon+1\le (\alpha+\epsilon)\deg(\frak b)=z
$$
(recall that $\deg(\frak b)\ge 3$ when $e\ge0$). Thus, if we start our construction of $\cF$ from a set $Z$ of $z$ points not lying on any divisor in the classes of $(\alpha+\epsilon) C_0$, then again
$$
h^0\big(S,\cI_{Z\vert S}(A+D)\otimes\mathcal L^\vee\big)\le \mathrm{max}\{0, h^0\big(S,\cO_S(A+D)\otimes\mathcal L^\vee\big) - z\}=0.
$$
%
\end{proof}

%
%
%

\end{document}